\def\draft{n}
\documentclass{amsart}
\usepackage[headings]{fullpage}
\usepackage{amssymb,epsfig,amsbsy,amsmath}
\theoremstyle{plain}

\newtheorem{theorem}{Theorem}
\newtheorem{proposition}{Proposition}[section]

\newtheorem{conjecture}{Conjecture}

\theoremstyle{definition}

\newtheorem{problem}{Problem}

\newcommand{\be}{\begin{equation}}
\newcommand{\ee}{\end{equation}}
\newcommand{\ba}{\begin{aligned}}
\newcommand{\ea}{\end{aligned}}
\theoremstyle{remark}

\newtheorem{remark}[proposition]{Remark}

\def\printname#1{
        \if\draft y
                \smash{\makebox[0pt]{\hspace{-0.5in}
                        \raisebox{8pt}{\tt\tiny #1}}}
        \fi
}

\newcommand{\psdraw}[2]
         {\begin{array}{c} \hspace{-1.3mm}
        \raisebox{-4pt}{\epsfig{figure=draws/#1.eps,width=#2}}
        \hspace{-1.9mm}\end{array}}

\newlength{\standardunitlength}
\catcode`\@=11
\long\def\@makecaption#1#2{%
     \vskip 10pt

\setbox\@tempboxa\hbox{%\ifvoid\tinybox\else\box\tinybox\fi
       \small\sf{\bfcaptionfont #1. }\ignorespaces #2}%
     \ifdim \wd\@tempboxa >\captionwidth {%
         \rightskip=\@captionmargin\leftskip=\@captionmargin
         \unhbox\@tempboxa\par}%
       \else
         \hbox to\hsize{\hfil\box\@tempboxa\hfil}%
     \fi}
\font\bfcaptionfont=cmssbx10 scaled \magstephalf
\newdimen\@captionmargin\@captionmargin=2\parindent
\newdimen\captionwidth\captionwidth=\hsize
\catcode`\@=12

\newcommand{\tr}{\operatorname{tr}}

\def\lbl#1{\label{#1}\printname{#1}}

\def\BN{\mathbb N}
\def\BZ{\mathbb Z}
\def\BP{\mathbb P}
\def\BQ{\mathbb Q}
\def\BR{\mathbb R}
\def\BC{\mathbb C}

\def\CF{\mathcal F}
\def\D{\Delta}

\def\a{\alpha}

\def\La{\Lambda}
\def\l{\lambda}
\def\Ga{\Gamma}

\def\ga{\gamma}

\def\la{\langle}
\def\ra{\rangle}

\def\e{\epsilon}
\def\Ga{\Gamma}
\def\d{\delta}
\def\b{\beta}

\def\Sym{\mathrm{Sym}}

\def\hb{\hbar}

\def\Tr{\mathrm{Tr}}

\def\fg{\mathfrak{g}}

\def\CA{\mathcal A}

\def\calH{\mathcal H}
\def\CI{\mathcal I}
\def\CO{\mathcal O}

\def\be{\begin{equation}}
\def\ee{\end{equation}}
\def\re{{\rm e}}
\def\calL{\mathcal{L}}
\def\calC{\mathcal{C}}
\def\diag{\mathrm{diag}}
\def\ep{\epsilon}

\newcommand{\ri}{{\rm i}}
\newcommand{\rd}{{\rm d}}

\begin{document}

%%%%%%%%%%%%%%%%%%%%%%{page1}

\title[Universality and asymptotics 
of graph counting problems in non-orientable surfaces]{
Universality and asymptotics of graph counting problems in non-orientable surfaces}
\author{Stavros Garoufalidis}
\address{School of Mathematics \\
         Georgia Institute of Technology \\
         Atlanta, GA 30332-0160, USA \\
         {\tt http://www.math.gatech} \newline {\tt .edu/$\sim$stavros } }
\email{stavros@math.gatech.edu}
\author{Marcos Mari\~no}
\address{Section de math\'ematiques \\
         Universit\'e de Gen\`eve \\
         CH-1211 Gen\`eve 4, Switzerland}
\email{Marcos.Marino@unige.ch}

\thanks{S.G. was supported in part by NSF. M.M. was supported in part by the Fonds National Suisse.
\newline
1991 {\em Mathematics Classification.} Primary 57N10. Secondary 57M25.
}

\date{January 6, 2009}

%\dedicatory{\large{\bf Preliminary version. Private copy.}}

\begin{abstract}
Bender-Canfield showed that a plethora of graph counting problems in 
orientable/non-orientable surfaces involve two constants $t_g$ 
and $p_g$ for the orientable and the non-orientable case, respectively. 
T.T.Q. Le and the authors recently
discovered a hidden relation between the sequence $t_g$ and a formal power
series solution $u(z)$
of the Painlev\'e I equation which, among other things, allows
to give exact asymptotic expansion of $t_g$ to all orders in $1/g$ 
for large $g$. The paper introduces
a formal power series solution $v(z)$ of a Riccati equation, gives a nonlinear
recursion for its coefficients and an exact asymptotic expansion
to all orders in $g$ for large $g$, using the theory of Borel transforms.
In addition, we conjecture a precise relation between the sequence $p_g$ and 
$v(z)$. Our conjecture is motivated by the enumerative aspects  
of a quartic matrix model for real symmetric matrices, and the 
analytic properties of its double scaling limit. In particular, the matrix
model provides a computation of the number of rooted quadrangulations
in the 2-dimensional projective plane. 
Our conjecture implies analyticity of the $\mathrm{O}(N)$ and 
$\mathrm{Sp}(N)$-types of free 
energy of an arbitrary closed 3-manifold in a neighborhood of zero.
Finally, we give a matrix model calculation of the Stokes constants, pose
several problems that can be answered by the Riemann-Hilbert approach,
and provide ample numerical evidence for our results.
\end{abstract}

\maketitle

\tableofcontents

%%%%%%%%%%%%%%%%% the text file

\section{Introduction}
\lbl{sec.intro}

\subsection{Counting rooted maps in orientable surfaces}
\lbl{sub.pg}

This paper introduces and studies the asymptotics of 
a sequence of rational numbers $(v_n)$ and a conjecture relating them
to a sequence of constants $(p_g)$ that appear in a rooted graph counting
problem of \cite{BC}. 
The problem of counting the number of graphs with a fixed number of edges
that can be embedded in a surface of genus $g$ has a long history. For planar
graphs, the problem was solved by Tutte; \cite{Tu}. In their seminal paper 
\cite{BC}, Bender-Canfield consider the number $T_g(n)$ of {\em rooted maps} 
$(G,S)$,
that is embeddings of a graph $G$ with $n$ edges in a closed connected orientable
surface $S$ of genus $g$, such that every component of $S\setminus G$ is a 
disk, and such that an edge, an orientation and a side of it is chosen.
Bender-Canfield gave an inductive computation of the natural number $T_g(n)$
and also proved that for fixed $g$ and large $n$, $T_g(n)$ is asymptotic
to
\begin{equation}
\lbl{eq.Tgn}
T_g(n) \sim t_g n^{\ga(g-1)} \l^n
\end{equation}
where
\begin{equation}
\lbl{eq.gala}
\ga=\frac{5}{2} \qquad \l=12
\end{equation}
and the constants $t_g$ are computable by some complicated
non-linear recursion that depends on an auxiliary partition; see 
\cite[Eqn.4.2]{BC} for a non-linear recursion for $\hat{\phi}^{(t)}_g(I,\a)$
that computes $t_g$ via \cite[Eqn.4.1]{BC}. The recursion of Bender-Canfield,
although cumbersome, gives exact answers for $t_g$ that do not depend on any
unknown constants. In particular, Bender-Canfield obtained the first
three values
$$
t_0=\frac{2}{\sqrt{\pi}}, \qquad t_1=\frac{1}{24}, \qquad
t_2=\frac{7}{4320 \sqrt{\pi}}.
$$
In the early nineties, it was realized in combinatorial enumeration that
several classes of counting problems (such as cubic maps, quartic maps)
also lead to an asymptotic expansion of the form \eqref{eq.Tgn}; 
see for example \cite{G1,G2}.

\subsection{Counting ribbon graphs in orientable surfaces}
\lbl{sub.ribbon}
In the physics literature, graphs appear often as Feynman diagrams of a 
perturbative quantum field theory. In the case of gauge theories with
gauge group $\mathrm{U}(N)$, the Feynman diagrams are {\em ribbon graphs}, 
i.e.,
graphs with a cyclic order of flags around each vertex. Perturbative
gauge theory counts ribbon graphs with a fixed genus, number of 
edges and prescribed valency, and with weight being the inverse of the size
of their automorphism group. This is precisely the content of matrix models
and was discussed extensively in the eighties, see \cite{BIPZ,BIZ}. One of 
the results found in this period in the matrix model 
community was 
that the generating functions counting ribbon graphs at fixed genus are 
analytic functions with a finite radius of convergence, which 
is the same for all genera. In the late eighties matrix model were studied 
in the so-called double scaling limit \cite{BK,DS,GM}, where roughly speaking 
one considers the generating functions near their singularity and extracts 
the coefficients $u_g$ of the leading poles for different genera. 
It was realized that these coefficients are {\em universal}, i.e., they do 
not depend on the details of the matrix model potential. It was also realized 
that
the ribbon counting problem for arbitrary potential gives rise to a generating
series in infinitely many variables that satisfies some {\em universal} (e.g.
KdV or KP) {\em hierarchy}. In addition, the double scaling limit of the
matrix model satisfies a non-linear differential equation which in the simplest
case is the famous {\em Painlev\'e I} equation; see for example \cite{DGJ,Wi} 
for a survey of these developments in the physics literature.

Two years ago, Goulden-Jackson proved that a similar generating
series that counts maps is a solution to the KP hierarchy, and in particular
deduced a quadratic nonlinear recursion for $T_g(n)$ in terms of $T_{g'}(n')$
for suitable $(g',n') < (g,n)$; see 
\cite[Eqn.45]{GJ}. Using the quadratic recursion of $T_g(n)$, in \cite{BGR}
Bender-Gao-Richmond give
a quadratic recursion relation for the constant $t_g$ that involves only
$t_{g'}$ for $g' < g$. This allowed them to prove
that $t_g/g!^2$ grows exponentially with a numerical non-zero
constant \cite{BGR}. Recently, it was realized in \cite{GLM}
that the quadratic recursion
relation for $t_g$ is, in disguise, the recursion characterizing a formal 
power series 
solution to Painlev\'e I. Although this (as well as Goulden-Jackson's work 
\cite{GJ}) came as a surprise to the enumerative
combinatorics community, it is hardly a 
surprise from the physics point of view. The fact that the $t_g$ enter into 
different map counting problems \cite{G1} is, from the 
matrix model point of view, a manifestation of universality. 

The analytic structure of solutions to Painlev\'e I is well-known, mostly 
though the Riemann-Hilbert approach (see for example \cite{FIKN}), and allows
one to give the full exact asymptotic expansion of $t_g$ in inverse
powers of $g$; see for example \cite[App.A]{GLM}.

\subsection{The case of non-orientable surfaces}
\lbl{sub.unor}

The above discussion focused on the counting problems of rooted maps and 
ribbon graphs in orientable surfaces. Although the problems are different,
their view from a distance (in the double-scaling limit)
is the same, described by a universal non-linear differential equation,
Painlev\'e I. On the combinatorial side, Bender-Canfield had also developed a 
theory of counting rooted maps in non-orientable surfaces. Let $P_g(n)$ be the number of $n$-edged rooted maps
 on a non-orientable surface of 
type $g$, where 
\be
g=1-{1\over 2} \chi.
\ee
Notice that, in the non-orientable case, $g$ can be integer or half-integer. Then, one has the asymptotics \cite{BC}
\be
P_g(n) \sim p_g n^{\gamma (g-1)} \lambda^n, \qquad g>0. 
\ee
This defines a sequence of constants $p_g$. On the physics side, the 
Feynman diagrams for $\mathrm{O}(N)$ or $\mathrm{Sp}(N)$ theories are ribbon 
graphs with {\em crosscaps} that are embedded in {\em non-orientable surfaces}.
This is discussed in detail in Section \ref{sec.LMO}. 
If one believes in a matching between the counting problems of
combinatorics and matrix models, the constants $p_g$
ought to be able to be computed and asymptotically analyzed by the
double scaling limit of an $\mathrm{O}(N)$ or $\mathrm{Sp}(N)$ matrix model. 
The latter is a 
pair of functions that satisfies a coupled system of two nonlinear differential
equations, as was explained in detail by Br\'ezin--Neuberger \cite{BN} and 
Harris--Martinec \cite{HM}.

Before we get into details, let us mention that the constants $p_g$ are 
notoriously hard to compute. In \cite[Eqn.3.6]{BC} Bender-Canfield give
a non-linear recursion for $\phi^{(t)}_g(I,\a)$
that computes $p_g$ via \cite[Eqn.3.3]{BC}. As in the case of $t_g$, 
the recursion is rather 
cumbersome and gives exact answers for $p_g$ that do not 
depend on any unknown constants. 
In particular, Bender-Canfield obtained the first
three values
\cite[p.245]{BC}
\be
\label{firstthree}
p_{1/2}=-\frac{2 \sqrt{6}}{\Gamma(-1/4)}, 
\quad 
p_{1}=\frac{1}{2},
\quad
p_{3/2}=\frac{\sqrt{6}}{3\Gamma(1/4)}.
\ee
The purpose of the present paper is to introduce
a sequence $(v_n)$ that is easy to compute via a quadratic nonlinear recursion
relation, and whose asymptotic is easy to analyze, and conjecturally agrees
with $p_g$ when $g=(n+1)/2$. Using the recursion of Bender-Canfield, Gao was able to match our
conjecture for $p_g$ for the first six values of $g$. Beyond that, 
nothing is known. The motivation for the sequence $(v_n)$ comes
from a study of the double scaling limit of $\mathrm{O}(N)$ matrix models
following Br\'ezin--Neuberger \cite{BN} and Harris--Martinec \cite{HM}.
This is discussed at leisure in Section \ref{sec.BN}.

\subsection{The sequence $(u_n)$}
\lbl{sub.GLM}
Since the counting problems of non-orientable surfaces mix with those of the
orientable ones (studied in \cite{GLM}), in this section we give a brief 
review of Appendix A of \cite{GLM}. Consider a function $u(z)$ that 
satisfies the Painlev\'e I differential equation
\begin{equation}
\lbl{eq.u}
u^2-\frac{1}{6} u''=z.
\end{equation}
Consider the unique formal power series solution to \eqref{eq.u}
asymptotic to $z^{1/2}$ for large $z>0$
\begin{equation}
\lbl{eq.fu}
u(z)=z^{1/2}\sum_{n=0}^\infty u_n z^{-5n/2}.
\end{equation}
It follows that the sequence $(u_n)$ satisfies the following recursion 
relation
\begin{equation}
\lbl{eq.recu}
u_n = \frac{25(n-1)^2-1}{48} u_{n-1} - \frac{1}{2} \sum_{k=1}^{n-1} 
u_k u_{n-k}, \qquad u_0 = 1. 
\end{equation}
It was observed that the sequences $(u_g)$ are $(t_g)$ are related by
\begin{equation}
\lbl{tgrel}
t_g=-\frac{1}{2^{g-2}\Gamma\left(\frac{5g-1}{2}\right)} u_g.
\end{equation}
To find the asymptotics of $(u_n)$ one uses a {\em trans-series} solution of 
the
differential equation \eqref{eq.u}. For a detailed discussion, see Section
\ref{sec.thm1} below. In our context, trans-series are mild generalizations
of formal power series and can be automatically computed much like the sequence
$(u_n)$ itself. The computation involves a finite number of unknown Stokes
constants (i.e., adiabatic invariants) $S$. 
In our case, we have that the sequence $(u_n)$ has an asymptotic
expansion of the form
\begin{equation}
\lbl{eq.asun}
u_n \sim A^{-2n+{1\over 2}} \, \Gamma\Bigl(2n-{1\over 2} \Bigr)\, 
{S \over 2 \pi \ri} \biggl\{1 + \sum_{l=1}^{\infty} {\mu_{l} A^{l} \over 
\prod_{m=1}^{l} (2n-1/2 -m)} \biggr\},
\end{equation}
where the so-called {\em instanton value} $A$ and the {\em Stokes constant} 
$S$ are given by 
\begin{equation}
\lbl{eq.AS}
A ={8 {\sqrt{3}} \over 5}, \qquad S = -\ri {3^{1\over 4} \over {\sqrt{\pi}}}, 
\end{equation}
and the $\mu_{l}$ are defined by the recursion relation 
\begin{equation}
\lbl{eq.mu}
\mu_{l}= {5\over 16 {\sqrt {3}} l} \biggl\{ {192 \over 25}  
\sum_{k=0}^{l-1} \mu_k u_{(l -k+1)/2} -\Bigl( l-{9\over 10} \Bigr) 
\Bigl( l-{1\over 10} \Bigr) 
\mu_{l-1} \biggr\}, \qquad \mu_0=1
\end{equation}
with the understanding that $u_{n/2}=0$ if $n$ is odd. To better understand
the recursion relation \eqref{eq.mu}, and to write it in a more compact form
that relates to the trans-series of \eqref{eq.u}, consider the generating 
series
\begin{equation}
\lbl{eq.u1}
u_1(z)=z^{1/2} z^{-5/8} e^{-A z^{5/4}} \sum_{n=0}^\infty \mu_n z^{-5n/4}.
\end{equation}
Then \eqref{eq.mu} is equivalent to the following linearized version of 
\eqref{eq.u}
\begin{equation}
\lbl{eq.linu}
u_1''-12 u u_1=0
\end{equation}
where $u(z)$ is given by \eqref{eq.u}. Hopefully, the reader will not confuse
the first term $u_1$ of the sequence $(u_n)$ with the function $u_1(z)$ 
in Equation \eqref{eq.linu}. 
%One can also study the asymptotics of the 
%sequence $(\mu_n)$, and so on. But we will not pursue this problem here.

\subsection{The sequence $(v_n)$}
\lbl{sub.vn}

In this section $u(z)$ will denote a function that satisfies the differential
equation \eqref{eq.u}.
Consider a function $v(z)$ that satisfies the differential equation

\begin{equation}
\lbl{eq.v}
%2 v^3-6v v'+2 v''-6uv+3u'=0.
2v'-v^2+3u=0.
\end{equation}
Consider the unique formal power series solution to \eqref{eq.v}
asymptotic to $z^{1/4}$ for large $z>0$
\begin{equation}
\lbl{eq.fv}
v(z)=z^{1/4} \sum_{n=0}^\infty v_n z^{-5n/4}.
\end{equation}
It follows that the sequence $(v_n)$ satisfies the following recursion relation
\begin{eqnarray}
\lbl{eq.recv}
v_n &=& 
\frac{1}{2 \sqrt{3}}\left(
-3u_{n/2}+\frac{5n-6}{2} v_{n-1} + \sum_{k=1}^{n-1} v_k v_{n-k}
\right), \qquad v_0=-\sqrt{3}
\end{eqnarray}
where $(u_n)$ is given by \eqref{eq.recu}, with the understanding that
$u_{n/2}=0$ if $n$ is not even.
Our theorem concerns the asymptotics of $(v_n)$ for large $n$.

\begin{theorem}
\lbl{thm.1}
The sequence $(v_n)$ has an asymptotic expansion of the form
\begin{equation}
\lbl{eq.asvn}
v_n \sim (A/2)^{-n} \Gamma(n) {S' \over 2\pi \ri} 
\biggl\{1 + \sum_{l=1}^{\infty} \frac{\nu_{l} (A/2)^{l}}{  \prod_{m=1}^{l} 
(n-m)} \biggr\}
\end{equation}
where $A$ is given in \eqref{eq.AS}, $S' \neq 0$ is some nonzero Stokes 
constant, and the sequence $(\nu_n)$ is defined by the
recursion relation
\begin{equation}
\lbl{eq.recnu}
\nu_n=-\frac{4}{5n} \sum_{k=0}^{n-1} v_{n+1-k} \nu_k, \qquad \nu_0=1.
\end{equation}
%
%\be
%\lbl{recnu}
%\ba
%\nu_l&= {1\over 20 {\sqrt{3}} l}\biggl\{ (5l-2)(5l-6) \nu_{l-1}\\
%&\qquad \qquad + \sum_{m=0}^{l-1}\nu_m \Bigl( (60l-24) v_{l-m} 
%+ 48 {\sqrt{3}} v_{l+1-m} + 48 \sum_{p=0}^{l+1-m} v_p v_{l+1-p-m} 
%- 48 u_{l+1-m\over 2} \Bigr) \biggr\},
%\ea
%\ee
% and $\nu_0=1$. 
%In Equation \eqref{eq.recnu} it is understood that $u_{k/2}=0$ if $k$ 
%is not an even integer. 
\end{theorem} 
An even more compact form of the recursion relation \eqref{eq.recnu} can
be given by introducing the generating series
\begin{equation}
\lbl{eq.v1}
v_1(z)= z^{1/4}  \sum_{n=0}^\infty \nu_n z^{-5n/4}.
\end{equation}
Then the recursion relation \eqref{eq.recnu} is equivalent to the following
linearized form of \eqref{eq.v}
%%%%%%%%%% See Mathematica file: ONGravity.nb
\begin{equation}
\lbl{eq.linv1}
v_1'=v v_1
\end{equation}
where $v(z)$ is given by \eqref{eq.fv}.

\subsection{Two conjectures}
\lbl{sub.conj}
We can now formulate two conjectures which are motivated by 
Section \ref{sec.BN} below. Our first conjecture links the sequence 
$(v_n)$ with the sequence $p_g$.
\begin{conjecture}
\lbl{conj.1}
For all $n=0,1,2,\dots$ we have 
\begin{equation}
\lbl{eq.pv}
p_{\frac{n+1}{2}}= 
\frac{1}{2^{\frac{n-3}{2}}\Gamma\left(\frac{5n-1}{4}\right)} v_n.
\end{equation}
\end{conjecture}

This conjecture reproduces the first three values (\ref{firstthree}) obtained 
in \cite{BC} and 
predicts for the next few values
\be
p_2=\frac{5}{36 \sqrt{\pi }},\, \, 
p_{5/2}=\frac{1033}{1024 \sqrt{6} \Gamma
   \left(\frac{19}{4}\right)},\, \, 
p_3=\frac{3149}{442368},\, \, 
p_{7/2}=\frac{1599895}{294912 \sqrt{6} \Gamma
   \left(\frac{29}{4}\right)},\, \, 
p_4=\frac{484667}{560431872 \sqrt{\pi }}, \, \, \cdots
   \ee

\begin{conjecture}
\lbl{conj.2}
The Stokes constant $S'$ is given by:
\begin{equation}
\lbl{eq.stokes}
S'=\ri {\sqrt{6}}.
\end{equation}
\end{conjecture}

\begin{remark}
\lbl{rem.morestokes}
Theorem \ref{thm.1} and Conjecture \ref{conj.1} reveal a single Stokes 
constant $S'$ associated with the asymptotics of the sequence $(v_n)$.
A second Stokes constant is needed for the asymptotics of the $k$-instanton
expansion of $(v_n)$; see Theorem \ref{thm.munk} and Remark 
\ref{rem.secondS}.
\end{remark}

\section{Future directions}
\lbl{sec.future}

\subsection{Approaches to Conjectures \ref{conj.1} and \ref{conj.2}}
\lbl{sub.app}

In this section we discuss some potential approaches to Conjectures 
\ref{conj.1} and \ref{conj.2}. Recall that Goulden-Jackson construct a
solution to the KP hierarchy for a generating series associated to the
$\mathrm{U}(N)$ gauge group. 

\begin{problem}
\lbl{prob.1}
%Parallel to the presentation in Section \ref{sub.ribbon}, there is 
Construct a version of the KP hierarchy corresponding to the
the generating series of the $\mathrm{O}(N)$ gauge group.
\end{problem}
Perhaps the theory of zonal polynomials will play a role 
analogous to the Schur functions in generalizing the work of \cite{GJ} 
and \cite{O}. This integrable hierarchy without doubt will imply, as 
in \cite[Eqn.45]{GJ} a quadratic recursion relation for the number of 
rooted maps $P_g(n)$ in an non-orientable surface of Euler characteristic $2-2g$
for a half-integer $g$. 

\begin{problem}
\lbl{prob.2}
State and prove a quadratic recursion relation for $P_g(n)$ in terms of
$P_{g'}(n')$ for $(g',n') < (g,n)$.
\end{problem}
Given this recursion relation, one may obtain
a combinatorial proof of Conjecture \ref{conj.1} along the line of thought of
\cite{BGR}. A solution to problem \ref{prob.2} can be obtained by identifying
$P_g(n)$ with an expectation value of an $\mathrm{O}(N)$ matrix model 
and deduce the quadratic relation
from the so-called {\em pre-string equation} of the matrix
model; see for example \cite{MSW}.

Let us point out that the pair of functions $(u,v)$ have a Lax pair 
and a Riemann-Hilbert problem, as was explained in \cite{BN}. 
Moreover, it is well-known that every solution to the Riemann-Hilbert problem
with rational jump functions is meromorphic, see \cite{Mi}. In addition,
the Stokes constants are exactly calculable from a Riemann-Hilbert problem.
For numerous instances of this calculation, that includes the case of
the Painlev\'e equations, see \cite{FIKN} and also \cite{Ka}. 

\begin{problem}
\lbl{prob.3}
Give a solution to the Riemann-Hilbert problem of the pair $(u,v)$
and compute the Stokes constant $S'$ confirming Conjecture \ref{conj.2}.
\end{problem}

\subsection{Relation to algebraic geometry} 
\lbl{sub.ag}

Let  $\overline M_{g,n}$ be the Deligne--Mumford moduli space of Riemann 
surfaces of genus $g$ with $n$ punctures, 
and let $\psi_i$, $i=1, \cdots, n$ be the two-cohomology class defined by 
\be
\psi_i=c_1({\mathcal L}_i), \qquad i=1, \cdots, n, 
\ee
where ${\mathcal L}_i$ is the bundle over $\overline M_{g,n}$ whose fiber 
at $\Sigma_{g} \in \overline M_{g,n}$ is the cotangent space $T^*\Sigma_g$ at 
the $i$-th puncture. In his seminal paper \cite{Wi}, Witten explains
how to package the enumerative intersection theory of the moduli space
of curves into a generating function that ought to satisfy the KdV equation
and some initial conditions. The conjecture was subsequently proven by 
Kontsevich \cite{K}. Using the Witten--Kontsevich theorem \cite{Wi,K} 
together with the results of \cite{IZ} it is possible to show that the 
intersection number
\be
\langle \sigma_2^{3g-3} \rangle_g =\int_{\overline M_{g,3g-3}} 
\psi_1^2 \wedge \cdots \wedge \psi^2_{3g-3}, \qquad g\ge 1,
\ee
is related to the coefficients $u_n$ defined in (\ref{eq.recu}) as follows
\begin{equation}
\lbl{eq.lu}
{\langle \sigma_2^{3g-3} \rangle_g  \over (3g-3)!}=
-{4^g \over (5g-5) (5g-3)}u_g.
\end{equation}
A proof of Equation \eqref{eq.lu} is given in \cite[Sec.6]{IZ} using 
the ansatz \cite[Eqn.5.22]{IZ}. A proof of this ansatz is given in 
\cite[Thm.3.1]{GJV} and also in \cite{EYY}.

Through (\ref{tgrel}) one finds an algebro-geometric interpretation of 
$t_g$ in terms of intersection numbers on $\overline M_{g,n}$. 
This relation motivates the following problem.

\begin{problem}
\lbl{prob.4} 
Give an enumerative algebro-geometric definition of the numbers $v_n$
(or, equivalently, $p_g$).
%Have the numbers $v_n$ (or, equivalently, $p_g$) an 
%algebro-geometric interpretation, perhaps in terms of intersection 
%theory on the moduli space of non-orientable Riemann surfaces?
\end{problem}

\subsection{Acknowledgement}
The authors wish to thank J. Gao for an independent computation of the
first values of the constant $p_g$ using the recursion of Bender-Canfield
and Nick Halmagyi for 
computing independently the first values of the number of quadrangulations 
on the projective plane. M.M. would also like to thank Bertrand Eynard for useful discussions.

\section{Analyticity of the $\mathrm{O}(N)$ and $\mathrm{Sp}(N)$ free energy 
of a closed 3-manifold}
\lbl{sec.LMO}

An initial motivation for our work is the problem of the free energy of
a closed 3-manifold $M$. In \cite{GLM} it was shown that the $\mathrm{U}(N)$ 
version of the free energy $F^U_M(\tau,\hb)$ of $M$ has the form
$$
F^U_M(\tau,\hb)=\hb^{-2} \sum_{g=0}^\infty \hb^{2g} F^U_{M,g}(\tau)
$$
where $\tau=N\hbar$ and 
$F^U_{M,g}(\tau) \in \BQ[[\tau]]$ are formal power series analytic in
a disk $D_M$ independent of $g$ that contains $0$ and depends on $M$.
$F^U_M$ is defined to be the logarithm of the LMO invariant, evaluated under
the $\mathrm{U}(N)$ weight system. One may also define the $\mathrm{O}(N)$ 
(resp. $\mathrm{Sp}(N)$) free
energy of a closed 3-manifold $M$ by
\begin{equation}
\lbl{eq.OSPfree}
F^O_M(\tau,\hb)=W_{\mathfrak{o}_N}(\log(Z_M)), 
\qquad
 F^{Sp}_M(\tau,\hb)=W_{\mathfrak{sp}_N}(\log(Z_M))
\end{equation}
where $Z_M$ is the LMO invariant of $M$ and $W_{\fg}$ denotes the weight system
of a metrized Lie algebra $\fg$; see \cite{B-N}. 
The $\mathrm{O}(N)$ (and also the $\mathrm{Sp}(N)$) free energy of a closed
3-manifold $M$ can be written in the form
\begin{equation}
F^O_M(\tau,\hb)=\hb^{-2} \sum_{g=0}^\infty \hb^{g} F^O_{M,g}(\tau)
\end{equation} 
where $\tau=N\hb$ and $F^O_{M,g}(\tau) \in \BQ[[\tau]]$ for all $g$.
In \cite[Rem.4.2]{GLM} it was observed that if $p_g$ satisfies an
asymptotic expansion similar to $t_g$, then the $\mathrm{O}(N)$ and 
$\mathrm{Sp}(N)$ 
free
energy of $M$ enjoys the same analyticity property as the $\mathrm{U}(N)$-free 
energy. Thus, Conjecture \ref{conj.1} implies that the $\mathrm{O}(N)$ and 
$\mathrm{Sp}(N)$ free
energy of $M$ is analytic in the above sense. 

Let us discuss in detail the weight system for the Lie algebra $\mathfrak{o}_N$
of the orthogonal group $\mathrm{O}(N)$. This explains the appearance of 
graphs 
in non-orientable surfaces. Let $e_{ij}$ denote the square matrix of size
$N$ with $ij$ entry equal to $1$ and all other entries zero. A basis for 
$\mathfrak{o}_N$ is $N_{ij}=e_{ij}-e_{ij}$ for $1 \leq i < j \leq N$. 
Since $e_{ij} e_{kl}=\d_{j,k} e_{il}$ (where $\d_{a,b}=0$ (resp. $1$) 
for $a=b$ (resp. $a \neq 1$)), and 
$\Tr(e_{ij})=\d_{i,j}$, (where $\Tr(M)$ denotes the trace of a square matrix 
$M$), it follows that the {\em Killing form} on $\mathfrak{o}_N$ is given
by
\begin{equation}
\lbl{eq.killinoN}
(N_{ij}, N_{kl})=2 (\d_{jk}\d_{il}-\d_{jl}\d_{ik})
\end{equation}
and its inverse (the so-called {\em propagator}) is given by
\begin{equation}
\lbl{eq.propoN}
\la N_{ij}, N_{kl} \ra=\frac{1}{2} (\d_{jk}\d_{il}-\d_{jl}\d_{ik}).
\end{equation}
This translates to the following diagrammatic way for the
$\mathfrak{o}_N$ weight system of a vertex-oriented cubic graph
(see for example \cite{B-N} and compare also
with \cite[Eqn.18]{GLM}):
$$
\psdraw{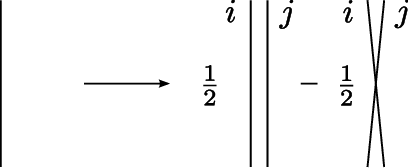}{2in}
$$
The resulting graphs are cubic 
ribbon graphs with cross-caps, and give rise exactly
to embedded trivalent graphs in non-orientable surfaces. 

In Section \ref{sec.BN} we will study the constants $p_g$ using not the
adjoint, but the symmetric representation 
of $\mathfrak{o}_N$. To explain this,
observe that if $\BC^N$ denotes the fundamental representation of 
$\mathfrak{o}_N$, then $\wedge^2(\BC^N)$ is the adjoint representation
of $\mathfrak{o}_N$ and $\Sym^2(\BC^N)$ can be identified with the set of
symmetric matrices of size $N$, with a basis given by 
$$
M_{ij}=e_{ij}+e_{ji}, \qquad \text{for} \qquad
1 \leq i \leq j \leq N.
$$ 
In that case, the symmetric matrix model 
\eqref{partitionf} of Section \ref{sec.BN} when $\b=1/2$ has propagator
given by
\begin{equation}
\lbl{eq.propoSN}
\la M_{ij}, M_{kl} \ra=\frac{1}{2} (\d_{jk}\d_{il}+\d_{jl}\d_{ik})
\end{equation}
and leads to the following diagrammatic way 
%for a quartic graph
$$
\psdraw{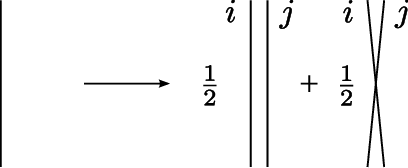}{2in}
$$

\section{The method of Borel transform}
\lbl{sec.borel}

The aim of this section is to prove the existence of an effective asymptotic 
expansion for the coefficients of a non-linear Euler-type differential 
equation using the theory of Borel transforms; see Theorem \ref{thm.munk} 
below. This will provide a proof of Theorem \ref{thm.1}.
In fact, Theorem \ref{thm.1} requires only a small portion of the theory, 
since the function $v$ of Theorem \ref{thm.1} satisfies a Riccati type 
equation \eqref{eq.v} that is well-studied in an excellent exposition of 
\cite{SS} and \cite[Sec.5]{BSS}. Nonetheless, the method of Borel transform is 
rather general and perhaps not as widely known. For the benefit of the reader,
we will give a short introduction to this beautiful theory
in the case of the {\em nonresonant Euler-type} differential equation
\begin{equation}
\lbl{eq.ODEf}
y'=-\l y -\frac{1}{x} \b y +g(x,y)
\end{equation}
where $g(x,y)$ is analytic at $(\infty,0)$, $g(x,y)=O(y^2,x^{-2})$, 
and the nonresonance condition is $\l \neq 0$. We will call $\l$ and $\b$
the {\em eigenvalue} and the {\em exponent} of \eqref{eq.ODEf}. $\l$ and $\b$
are precisely the coefficients of the linearized equation
\begin{equation}
\lbl{eq.ODEliny}
y'= -\l y -\frac{1}{x} \b.
\end{equation}
The Euler-type differential equation includes as a special case
the {\em Riccati-type} differential equation
\begin{equation}
\lbl{eq.ricf}
y'=-\l y -\frac{1}{x} \b y +g_0(x) y + g_1(x) y^2
\end{equation}
where $g_0$ and $g_1$ are analytic at $x=\infty$ and $g_0(x)=O(x^{-2})$.  

It is easy to see that a nonresonant Euler-type differential equation 
\eqref{eq.ODEf} has a unique formal power series solution 
$f_0(x) \in \BC[[1/x]]$. 
In fact, Equation \eqref{eq.ODEf} has a unique {\em trans-series}
solution of the form $\hat{f}(x)$ where
\begin{equation}
\lbl{eq.transf}
\hat{f}(x)=\sum_{k=0}^\infty C^k f_k(x), \qquad f_k(x)=x^{-\b k} e^{-\l k x} 
\sum_{n=0}^\infty \mu_{n,k}\frac{1}{x^n}, \qquad \mu_{1,0}=1
\end{equation}
is obtained by substituting the above expression in \eqref{eq.ODEf}
and equating the coefficient of every power of $C$ to zero. This leads to
a hierarchy of differential equations for $f_k(x)$ which is nonlinear 
for $k=0$, linear homogeneous for $k=1$, and linear inhomogeneous for 
$k \geq 2$. For example, we have:
%%%%% See Mathematica file: ONGravity.nb 
\begin{equation}
\lbl{eq.linODEf}
f_1'=-\lambda f_1 + g_y(x,f_0) f_1
\end{equation}
The above hierarchy of differential
equations was first discovered by \'Ecalle and was studied in \cite[p.54]{CNP},
\cite{C}, \cite[p.11-12]{Sa1} and \cite[Sec.8]{Sa2}. An excellent presentation
for the special case of the {\em Riccati equation} was given in 
\cite[Sec.5]{BSS} and \cite{SS}. This special case is of interest to us
since $v(z)$ satisfies the Riccati equation \eqref{eq.v}. It is important
to realize that the coefficients $\mu_{n,k}$ are automatically computed from
\eqref{eq.ODEf} by means of a nonlinear recursion relation that involves
$\mu_{n',k'}$ for $(n',k') < (n,k)$. 

The next theorem, presumably well-known to the experts
but absent from the literature, 
links the asymptotics of the trans-series coefficients $(\mu_{n,k})$ in
terms of the neighboring coefficients $(\mu_{l,k\pm 1})$ and two Stokes
constants $S_{\pm 1}$. In the physics literature, $f_k(x)$ often occurs as
the perturbation theory of the $k$-instanton solution, in which case Theorem
\ref{thm.munk} states that the asymptotics of the coefficients of the
$k$-instanton solution can be computed by the $(k\pm 1)$-instanton solutions 
and two adiabatic invariants $S_{\pm 1}$.

\begin{theorem}
\lbl{thm.munk}
Consider the differential equation \eqref{eq.ODEf} with $\l \neq 0$
and the coefficients $(\mu_{n,k})$ of the unique trans-series 
solution \eqref{eq.transf}. Then, for every $k=0,1,2,\dots$ and $n$
large we have an asymptotic expansion 
\begin{eqnarray}
\lbl{eq.asall}
\mu_{n,k} & \sim & \l^{-n+\b} (k+1) {S_1 \over 2\pi \ri} 
\Gamma\bigl(n-\b \bigr)\, 
\biggl\{\mu_{0,k+1} 
+ \sum_{l=1}^{\infty} {\mu_{l,k+1} \l^{l} \over \prod_{m=1}^{l} 
(n-\b -m)} \biggr\} \\ 
\notag & + &
(-\l)^{-n-\b} (k-1){S_{-1} \over 2\pi \ri} \Gamma\bigl(n+\b \bigr)\, 
\biggl\{\mu_{0,k-1}
 + \sum_{l=1}^{\infty} {\mu_{l,k-1} (-\l)^{l} \over \prod_{m=1}^{l} 
(n+\b -m)} \biggr\}
\end{eqnarray}
with the understanding that $\mu_{n,k}=0$ for $k<0$. $S_{\pm 1}$ are two
Stokes constants, defined below. 
\end{theorem}

\begin{proof}
We will give the proof in three steps. At first, let us suppose that
$k=0$ and let us denote $\mu_{n,0}=a_n$. In other words, we have
\begin{equation}
\lbl{eq.seriesf}
f_0(x)=\sum_{n=1}^\infty a_n \frac{1}{x^n} \in \BC[[x^{-1}]]
\end{equation}
is the unique formal power series solution of \eqref{eq.ODEf}. 
The first step provides the
existence of an asymptotic expansion of $(a_n)$.

{\bf Step 1}: With the above assumptions,
the sequence $(a_n)$ has an asymptotic expansion of the form
\begin{equation}
\lbl{eq.asexp}
a_n \sim \l^{-n+\b} {S_1 \over 2 \pi \ri}  \Gamma\Bigl(n-\b \Bigr)\, 
\biggl\{1 + \sum_{l=1}^{\infty} {\mu_{l} \l^{l} \over \prod_{m=1}^{l} 
(n-\b -m)} \biggr\}
\end{equation}
where $S_1$ is an unspecified Stokes constant.

A proof of Step 1 is given in  \cite[Lem.1.1]{CK1}. 
The proof of Lemma 1.1 of \cite{CK1} starts by writing down a
nonlinear recursion relation for the rescaled coefficients $b_n=\l^n a_n/n!$
of a formal power series solution in the form
$$
b_n=b_{n-1}+R_n(b_0,\dots,b_{n-1})
$$
where $R_n(b_0,\dots,b_{n-1})$ is a nonlinear term, a polynomial in the 
variables $b_0,\dots,b_{n-1}$. Then, one obtains an estimate of the nonlinear 
term of the form $R_n(b_0,\dots,b_n)=O(1/n^2)$. This implies that
$$
b_n=b_{n-1}+O\left(\frac{1}{n^2}\right)
$$
It follows that $(b_n)$ is increasing and bounded above. Thus, the limit
of $b_n$ exists. Using this as input, one bootstraps and obtains, order
by order in powers of $1/n$ an asymptotic expansion of $(b_n)$. 

The next step identifies the coefficients $\mu_l$
of the asymptotic expansion of $(a_n)$ with the coefficients of the
first trans-series $f_1(x)$ of \eqref{eq.ODEf}.

{\bf Step 2}: With the assumption of Step 1, we have 
\begin{equation}
\lbl{eq.mu1}
\mu_n=\mu_{n,1}
\end{equation}
for all $n$. This proves Theorem \ref{thm.munk} when $k=0$.

A proof of Step 2 appears implicitly in \cite[p.1939]{CK2} and also
in \cite{CNP, BSS,SS,Sa1,Sa2}. 
Let us take this opportunity to sketch the argument for 
completeness and also to fix several typographical errors of 
\cite[p.1939]{CK2}. Recall the notion of {\em Borel transform} which sends
power series in $1/x$ to power series in $p$:
$$
\frac{1}{x^{\a}} \mapsto 
\frac{p^{\a-1}}{\Ga(\a)}.
$$
Consider the differential equation \eqref{eq.ODEf}
with unique formal power series solution $f_0(x)$, and let $\phi_0(p)$
denote the Borel transform of $f_0(x)$ given by
\begin{equation}
\lbl{eq.borelt}
f_0(x)=\sum_{n=0}^\infty a_n \frac{1}{x^n}, \qquad
\phi_0(p)=\sum_{n=0}^\infty \frac{a_{n+1}}{n!} p^n.
\end{equation}
Step 1 implies that $\phi_0(p)$ is analytic at $p=0$.
Consider the punctured plane $W=\BC\setminus \calL$ where
$\calL=\l \BN^+$ is a discrete subset of punctures that
lies in a ray $\l [1,\infty)$. Consider the universal cover $\widetilde{W}$
which is identified with the set of homotopy classes (rel. boundary) of
paths that begin at $0$ and lie in $W$. Of importance is the portion 
$W^0 \cup W^1$ of $W$ where
\begin{itemize}
\item[(a)]
$W^0$ is the first Riemann sheet of $W$ given by $W^0=
\BC\setminus \l[1,\infty)$. In other words, $W^0$ is the plane minus one cut.
$$
\psdraw{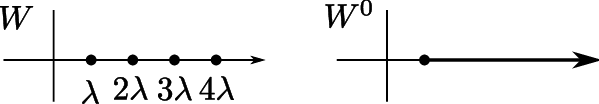}{3in}
$$
\item[(b)]
$W^1$ is the second Riemann sheet of $W$ given by all paths in $W$ that
cross the cut $\l[1,\infty)$ at most once.
\end{itemize}
In \cite[Thm.A.1]{CK1} it is shown that $\phi_0$ admits analytic 
continuation as an analytic function in $W^0 \cup W^1$ and $\phi_0(p)$
has local expansion around the singularity $p=\l$ of the form 
\begin{equation}
\lbl{eq.zetaphi}
\D_{\l} \phi_0 = S_1 \phi_1
\end{equation}
where $S_1 \in \BC$ and 
\begin{equation}
\lbl{eq.borelf1}
\phi_1(p)=p^{\b-1} \sum_{l=0}^\infty \frac{\mu_{l,1}}{\Ga(\b+l)} p^l
\end{equation}
is the Borel transform of the trans-series solution $f_1(x)$ 
and 
\begin{equation}
\lbl{eq.aliend}
\D_{\mu}g(p)=\lim_{\e\to 0^+}  g(\mu+p e^{i \e})-g(\mu+p e^{-i \e})
\end{equation}
denotes the {\em variation} of a multivalued analytic germ $g$ near
a singularity $p=\mu$; the so-called {\em alien derivative} in \'Ecalle's
language \cite{Ec}.

A well-known application of Cauchy's theorem in Borel plane 
(see \cite[p.1939]{CK2} and also \cite[Sec.7]{CG1}) 
now implies the asymptotic expansion \eqref{eq.asexp}
of $(a_n)$. Let us sketch the argument, fixing some typographical errors from
\cite[p.1939]{CK2}, and giving an exact formula, well-known to the
physics community. 
Cauchy's theorem for $\phi_0(p)$ together with 
\eqref{eq.borelt} and the analyticity of $\phi_0(p)$ at $p=0$ implies that
$$
\frac{a_{n+1}}{n!} =\frac{1}{2 \pi \ri} \int_{\ga} \frac{\phi_0(p)}{p^{n+1}}
\rd p
$$
where $\gamma$ is a small circle around $p=0$. Now enlarge the contour $\ga$
to a contour $\calC \cup \calH_{\l}$ in the first sheet $W^0$, where 
$\calC$ is a circle of radius $|\l|+\e$, minus an arc, and $\calH_{\l}$ is
a Hankel contour centered at $\l$, for some fixed $\ep$: 
$$
\psdraw{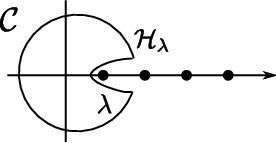}{2in}
$$
Thus,
$$
\frac{a_{n+1}}{n!} =\frac{1}{2 \pi \ri} \int_{\calC} \frac{\phi_0(p)}{p^{n+1}}
\rd p + \frac{1}{2 \pi \ri} \int_{\calH_{\l}} \frac{\phi_0(p)}{p^{n+1}}
\rd p 
$$
The first integral can be estimated by $O(n!(|\l|+\ep)^{-n})$. For the
second integral, make a change of variables $p=\l(1+z)$ and let $\calH$ denote
a Hankel contour centered around $0$. Then, we have
$$
\int_{\calH_{\l}} \frac{\phi_0(p)}{p^{n+1}} \rd p=
\l^{-n} \int_{\calH_0} \frac{\phi_0(\l(1+z))}{(1+z)^{n+1}} \rd z=
\l^{-n} \int_0^{\e} \frac{\D_{\l}\phi_0(\l z)}{(1+z)^{n+1}} \rd z
$$
Now, Equations \eqref{eq.zetaphi} and \eqref{eq.borelf1} imply that
$$
\D_{\l} \phi_0(\l z)
=S_1 \l^{\b-1}\sum_{l=0}^\infty \frac{\mu_{l,1} \l^l}{\Gamma(\b+l)} z^{\b+l-1}
$$
and the series on the right hand side, after multiplication by $p^{-\b+1}$,
is analytic at $p=0$. A useful {\em Beta-integral} calculation gives that
$$
\int_0^\infty \frac{z^{\ga-1}}{(1+z)^{n+1}} \rd z
=\frac{\Ga(\ga)\Ga(n+1-\ga)}{\Ga(n+1)},
$$
and therefore
$$
\int_0^{\ep} \frac{z^{\ga-1}}{(1+z)^{n+1}} \rd z
=\frac{\Ga(\ga)\Ga(n+1-\ga)}{\Ga(n+1)}(1+O((|\l|+\ep)^{-n}))
$$
%%%% See Mathematica file CostinKruskalSingularitiesBorelTransform.nb
Interchanging summation and integration by applying Watson's lemma, 
(see \cite{O}) it follows that
$$
a_{n+1} \sim \l^{-n+\b-1} \frac{S_1}{2 \pi \ri} 
\sum_{l=0}^\infty \Gamma(n-\b-l+1) \mu_{l,1} \l^l
$$
Comparing the above with Equation \eqref{eq.asexp} concludes Step 2.
Strictly speaking, the above analysis works only 
when $\Re(\b)>-1$. This is a local integrability assumption of the 
Beta-integral. The asymptotic expansion \eqref{eq.asexp} remains valid as 
stated even when $\Re(\b) \leq -1$ as follows by first integrating $f(x)$ 
a sufficient number of times, and then applying the analysis. This is exactly 
what was done in \cite{CK2} and \cite[Sec.7]{CG1}. 
%at the cost of complicating the notation.

Let 
$$
\phi_k(p)=p^{-k\b-1}\sum_{n=0}^\infty \frac{\mu_{n,k}}{\Ga(k\b+n)} p^n
$$ 
denote the Borel transform of $f_k(x)$ (given by \eqref{eq.transf}) for 
$k=0,1,2,\dots$.

{\bf Step 3}: For $k=0,1,2,\dots$, $\phi_k(p)$ has analytic continuation
as a multivalued analytic function in $\BC\setminus (-k\l+\BN^+\l)$.
Moreover, we have
\begin{equation}
\lbl{eq.alienf}
\D_{l \l} \phi_k= (k+l)S_{l} \phi_{k+l}
\end{equation}
for $l=1,-1,-2,-3,\dots$, $k=0,1,2,3,\dots$, $k+l=0,1,2,\dots$ where
$S_{l}$ are constants. Moreover \cite[Prop.5.4]{BSS}, 
in the case of the Riccati equation \eqref{eq.ricf}, we have $S_l=0$ for 
$l \leq -2$. Equation \eqref{eq.alienf}
is the so-called {\em Bridge Equation} of \'Ecalle, and appears in
\cite[p.63]{CNP} and also in \cite[p.11-12]{Sa1}. For a detailed
discussion, see also \cite[Prop.5.4]{BSS} and \cite{SS,Sa2}.

Said differently, the Bridge Equation implies that $\phi_0(p)$
admits analytic continuation as an analytic function in $W$
and the local expansion of every branch of $\phi_0(p)$ at each singularity
is of the form \eqref{eq.zetaphi} for a suitable Stokes constant $S_j$.
This is the notion of {\em resurgence} coined and studied systematically by 
\'Ecalle in the eighties. Unfortunately, \'Ecalle's work remains unpublished, 
and appears in three Orsay preprints that the interested reader 
is welcomed to read, \cite{Ec}. Fortunately for our purposes, the
Bridge equation for Euler-type differential equations 
has been established in print in the above mentioned references.

The Bridge Equation implies that
the nearest nonzero singularity of $\phi_k(p)$ appears at $p=\pm \l$
(resp. $p=\l$) when $k \geq 2$ (resp. $k=0,1$), and that the variation of
$\phi_k(p+\l)$ (resp. $\phi_k(p-\l)$) is proportional to $\phi_{k+1}(p)$
(resp. $\phi_{k-1}(p)$). Theorem \ref{thm.munk} follows for all $k$
by applying Cauchy's formula to $\phi_k(p)$, and 
deforming the contour of integration as in the case of $k=0$.
\end{proof}

Let us make some remarks.

\begin{remark}
\lbl{rem.singularODE}
As stated, Theorem \ref{thm.munk} requires analyticity of the coefficient
$$
g(x,y)=\sum_{n=-1}^\infty g_n(x) y^{n+1}
$$
of \eqref{eq.ODEf} at $(\infty,0)$. In fact, the key Equation 
\eqref{eq.zetaphi} and Theorem \ref{thm.munk} remains
true when $k=0$ and the Borel transform $\ga_n(p)$ of $g_n(x)$ has analytic
continuation as a multivalued analytic function in $\BC$ minus a discrete
set of points, and $\ga_n(p)$ is analytic in the disk $|p|<|\l|+\ep$
for all $n \in \BN$. 
This follows from the proof of the resurgence of $\phi_k(p)$ in Borel plane;
see \cite{CNP,Sa1,Sa2,BSS,SS}. In particular, Equation \eqref{eq.zetaphi}
and Theorem \ref{thm.munk} holds for $k=0$ when the $\ga_n(p)$ is
analytic multivalued in $\BC\setminus 2 \BZ^* \l$ for all $n \in \BN$, where 
$\BZ^*=\BZ\setminus\{0\}$. Theorem \ref{thm.munk} will be applied in this
form to give a proof of Theorem \ref{thm.1}.
\end{remark}

\begin{remark}
Starting from a sequence $(\mu_{n,0})$ with an asymptotic expansion
\eqref{eq.asall}, one can consider
the finitely many sequences $(\mu_{n,1})$ that arise, and repeat this
process. This may be continued for ever, arriving at a notion of resurgence
for sequences $(\mu_{n,0})$ that may not come from differential equations;
see for example \cite{CG1,CG2,Ga}.
\end{remark}
 
\begin{remark}
\lbl{rem.StokesS}
The constants $S_{\pm 1}$ in Equation \eqref{eq.asall}, if nonzero, can be 
numerically approximated rigorously and efficiently.
For examples, see \cite{CCK,JK,BGR}. On the other hand, deciding whether
$S_{\pm 1}=0$ is not algorithmically known. In general
the Stokes constants are transcendental invariants of the differential 
equation, sometimes known by the name of {\em adiabatic invariants}.
\end{remark}

\begin{remark}
\lbl{rem.vectorf}
Theorem \ref{thm.munk} is valid for systems of first order nonlinear 
differential 
equations of rank 1 under a nonresonance assumption, see \cite{C} for
Steps 1 and 2 and a comment on \cite[Sec.13]{Sa2}. In that
case the nonlinear equation is
\begin{equation}
\lbl{eq.vectorf}
{\bf y}'=-\La {\bf y} -\frac{1}{x} B {\bf y} +{\bf g}(x,{\bf y})
\end{equation}
and the linearized equation is
$$
{\bf y}'=-\La {\bf y} -\frac{1}{x} B {\bf y}
$$
with {\em eigenvalue matrix} $\La=\diag(\l_1,\dots,\l_r)$
and {\em exponent matrix} $B=(\b_1,\dots,\b_r)$, where $\diag(d_1,\dots,d_r)$
denotes a diagonal matrix with diagonal entries $d_1,\dots,d_r$.
The Borel transform ${\bf \phi_0}$ of the unique formal power series solution 
${\bf f_0}$ to 
\eqref{eq.vectorf} is a multivalued analytic function on 
$W=\BC\setminus \calL$ where $\calL=\l_1 \BN^+ + \dots \l_r \BN^+$. Moreover,
${\bf \phi_0}$ has analytic continuation in the principal sheet
$W^0=\BC\setminus \cup_{j=1}^r \l_j [1,\infty)$ which is a plane cut by
$r$ rays.
The asymptotic expansion of the coefficients of a component of a formal 
power series solution ${\bf f_0}$ of the nonlinear differential equation
has an asymptotic expansion of the form
\begin{equation}
\lbl{eq.total}
a_n \sim \sum_{j \in J} \l_j^{-n+\b_j} \frac{S_j}{2 \pi \ri}
\Gamma\Bigl(n-\b_j \Bigr)\, 
\biggl\{1 + \sum_{l=1}^{\infty} {\mu_{l,j} \l_j^{l} \over \prod_{m=1}^{l} 
(n-\b_j -m)} \biggr\}
\end{equation}
where $J=\{ j \in \{1,\dots,r\}\, | |\l_j|
=\min\{|\l_s| \, | \, s=1,\dots,r\}$.
\end{remark}
The analogue of Equation \eqref{eq.asall} does not seem to exist in the
literature, especially in the resonant case. For some partial results 
regarding the Painlev\'e I equation, see \cite{CK1}.

\section{Proof of Theorem \ref{thm.1}}
\lbl{sec.thm1}

\subsection{Theorem \ref{thm.munk} implies Theorem \ref{thm.1}}
\lbl{sub.thm1}

In this section we give a proof of Theorem \ref{thm.1} using 
Theorem \ref{thm.munk}. Consider the differential equation \eqref{eq.v} 
where $u(z)$ satisfies \eqref{eq.u}. The change of variables 
\begin{equation}
\lbl{eq.uvUV}
u(z)=z^{1/2}(1+U(x)), \qquad v(z)=z^{1/4}(-\sqrt{3}+V(x)), \qquad
x=z^{5/4}
\end{equation}
converts \eqref{eq.u} to the following rank-1 differential equation
%%% See ONGravity.nb 
\begin{equation}
\lbl{eq.ODEU}
U''(x)+\frac{U'(x)}{x} -\frac{96 U(x)^2}{25}-\frac{4 U(x)}{25 x^2}
-\frac{4}{25 x^2}=0
\end{equation}
which can be written as a first order rank-1 differential equation  
in the form
\begin{equation}
\lbl{eq.ODEU2}
{\bf U}'=-\hat\Lambda {\bf U} -\frac{1}{x} \hat B 
{\bf U} + {\bf \hat g}(1/x,{\bf U})
\end{equation}
where
\begin{equation}
{\bf U}=\left(\begin{matrix} U \\ U' \end{matrix}\right),
\qquad
\hat
\Lambda=\left(\begin{matrix} 0 & -1 \\ -\frac{192}{25} & 0 \end{matrix}\right),
\qquad
\hat B=\left(\begin{matrix} 0 & 0 \\ 0 & 1 \end{matrix}\right)
\end{equation}
and
$$
{\bf \hat g}(1/x,{\bf U})=
\left(\begin{matrix} 0 \\  \frac{96 U(x)^2}{25}+\frac{4 U(x)}{25 x^2}
+\frac{4}{25 x^2} \end{matrix}\right)
$$
A gauge transformation ${\bf U} \mapsto {\bf U} G$ converts \eqref{eq.ODEU2}
into the normalized resonant rank-1 differential equation
\begin{equation}
\lbl{eq.ODEU3}
{\bf U}'=-\Lambda {\bf U} -\frac{1}{x}  B 
{\bf U} + {\bf g}(1/x,{\bf U})
\end{equation}
where
\begin{equation}
\Lambda=\left(\begin{matrix} A & 0 \\ 0 & -A \end{matrix}\right),
\qquad
B=\left(\begin{matrix} \frac{1}{2} & 0 \\ 0 & \frac{1}{2} 
\end{matrix}\right)
\end{equation}
and $A=8 \sqrt{3}/5$ as in \eqref{eq.AS}. Observe that the differential
equation \eqref{eq.ODEU3} is resonant.
Nontheless, Theorem \ref{thm.1} involves the
function $v(z)$ which satisfies a Riccati equation \eqref{eq.v}. 
The substitution \eqref{eq.uvUV} converts \eqref{eq.v} to the following
rank 1 differential equation
\begin{equation}
\lbl{eq.ODEV1}
V'(x)+\frac{4 \sqrt{3}}{5} V(x)-\frac{\sqrt{3}}{5x}+\frac{6 U(x)}{5}
+\frac{V(x)}{5x}-\frac{2 V(x)^2}{5}=0
\end{equation}
This is a nonresonant Riccati differential equation of the form 
\eqref{eq.ricf} with eigenvalue $A/2=\frac{4 \sqrt{3}}{5}$ and exponent
$\beta=0$. Even though the coefficients of \eqref{eq.ODEV1} are not
analytic, their Borel transform is a multivalued analytic function in
$\BC\setminus \BZ^* A$ and Remark \ref{rem.singularODE} applies.
The trans-series solution of \eqref{eq.v} is of the form
\begin{equation}
\lbl{eq.transv}
v(z)=\sum_{l=0}^\infty C^l v_l(z)
\end{equation}
where $v_0(z)$ is given in \eqref{eq.fv} and 
$$
v_k(z)=z^{1/4} e^{-\frac{A}{2}z^{5/4}l} \sum_{n=0}^\infty v_{n,k}z^{-5n/4}
$$
and we normalize $v_{0,1}=1$. Theorem \ref{thm.1} uses only $v_1(z)$
which satisfies the linearized version 
$$
v_1'-v_0 v_1=0
$$
of \eqref{eq.v}, as stated in \eqref{eq.linv1}. 
Theorem \ref{thm.1} follows from Theorem \ref{thm.munk} when $k=0$
applied to the Riccati equation \eqref{eq.ODEV1}. This Riccati equation
does not have analytic coefficients in a neighborhood of $x=\infty$, however
Remark \ref{rem.singularODE} applies.
A rigorous numerical calculation of the Stokes constant $S'$ is possible,
following standard arguments that appear for example in \cite{CCK,JK}
to conclude that $S' \neq 0$. 
This concludes the proof of Theorem \ref{thm.1}.
\qed

\begin{remark} 
\lbl{rem.vk}
One can study in detail the full trans-series solution 
\eqref{eq.transv}. The $v_k$ satisfy the differential 
equations:
\be
\label{vkp}
v_k' -{1\over 2} \sum_{i=0}^k v_i v_{k-i}=0.
\ee
In principle, in writing the equation for the formal trans-series solution 
to $v$, we have to include in Equation \eqref{eq.v} a full formal trans-series 
solution for $u(z)$, but it is easy to see that the only way to find a 
solution for the $v_k$ is to set this trans-series to zero. From (\ref{vkp}) 
one easily deduces the following recursion relation for $v_{n,l}$:
\be
v_{n+1,k} =-{1\over {\sqrt{3}} (k-1)} \Bigl\{ {5n \over 4} v_{n,k} 
+ \sum_{l=2}^{n+1} v_{n+1-l,k}v_l
+{1\over 2} \sum_{i=1}^{k-1} \sum_{l=0}^{n+1} v_{l,i} v_{n+1-l,k-i} \Bigr\}, 
\qquad k\ge 2.
\ee
One finds, for example, 
\be
v_{0,k}=(-1)^{k-1} (2 {\sqrt{3}})^{1-k}, \quad k\ge 1.
\ee
For equations of the Riccati type, it was shown in \cite{BSS} that the full trans-series expansion can be written in terms of 
three formal power series, as follows. Let us denote
\be
\hat v_0(x) = \sum_{n=2}^{\infty} v_{n} x^{-n}, \qquad \hat v_k(x)= \sum_{n=0}^{\infty} v_{n,k} x^{-n}, \quad k\ge 1.
\ee
Then, there are formal power series $v_{\pm}(x)$ such that
\be
\hat v_k= (-1)^{k-1} v_+^{k-1} v_-^k (1-v_+ \hat v_0), \qquad k \ge 1. 
\ee
In our case we have 
\be
\ba
v_+(x)&=\frac{1}{2 \sqrt{3}}+\frac{5 }{192 \sqrt{3}}{1\over x^2}-\frac{25 }{1152}{1\over x^3}+\frac{3149 }{36864 \sqrt{3}}{1\over x^4}-\frac{15995
   }{110592}{1\over x^5}+\CO\left(x^{-6}\right),\\
v_-(x)&=1-\frac{1}{4 \sqrt{3}}{1\over x}-\frac{1}{24}{1\over x^2}-\frac{1459 }{11520 \sqrt{3}}{1\over x^3}-\frac{5429 }{34560}{1\over x^4}-\frac{114343 }{138240
   \sqrt{3}}{1\over x^5}+\CO\left(x^{-6}\right).
   \ea
   \ee

\end{remark}

\subsection{A brief discussion of the Riemann-Hilbert approach}
\lbl{sub.RH}

The Riemann-Hilbert method is an alternative way of proving and computing
asymptotic expansions of the form \eqref{eq.total}. For an example relevant
to the results of the paper, see \cite[App.A]{Ka} where Kapaev computes the
asymptotic expansion of the sequence $(u_n)$. The Riemann-Hilbert method 
uses the fact that all solutions of the Riemann-Hilbert problem (such 
as the function $u(z)$ and conjecturally also $v(z)$) of isomonodromy are
{\em meromorphic} 
functions in the $z$-plane, with prescribed behavior at various
sectors. Applying the Cauchy formula and a deformation of the contour argument
as in Claim 2 above, allows one to deduce the asymptotic expansion of the
coefficients $(a_n)$ of a formal solution; see for example \cite[App.A]{Ka}.
In addition, the Riemann-Hilbert approach computes exactly the corresponding
Stokes constants $S_j$ in \eqref{eq.total}. Whether the Riemann-Hilbert
method computes the coefficients $\mu_n$ in Equation \eqref{eq.asexp} via
the recursion \eqref{eq.mu} is unknown to us.

The Riemann-Hilbert approach overlaps with
(but is neither a subset or a superset of) the theory of Borel transform.
On the one hand, 
the differential equations that are amenable to the Riemann-Hilbert
method are very closely linked to the notion of {\em integrability}, whereas
the Borel transform method can deal with generic nonlinear differential
equations. On the other hand, the Riemann-Hilbert method can deal with
variational problems that do not come from differential equations. 

It was recently realized that the method of Borel transform is useful in
studying problems of physical origin (such as those originating in 
3-dimensional Quantum Topology) that do not come from differential/difference
equations. See for example the results of \cite{CG1,CG2} and the survey 
paper \cite{Ga}.

Without doubt, the Riemann-Hilbert method is complementary and overlapping
with the method of Borel transform, and the combination of the two methods
can give powerful results which have yet to be witnessed.

An alternative method to asymptotic expansions, that include exact computation
of the Stokes constants is available from physics. The part of this
theory, relevant to our problem, will be reviewed in Section \ref{sec.mm}.

\section{Matrix models, orthogonal ensembles and non-orientable graphs}
\lbl{sec.BN}

In this section we will give our motivation for Conjecture \ref{conj.1} 
following the study of the {\em symmetric quartic matrix model} studied by 
Br\'ezin--Neuberger \cite{BN} and Harris--Martinec \cite{HM}. 

After reduction to eigenvalues, the so-called {\it partition function} of 
a matrix model can be written as 
\be
\lbl{partitionf}
Z_{\beta} ={1\over (N!)^{\beta} (2 \pi)^{\beta N}} 
\int \prod_{i=1}^N \rd\lambda_i \, |\Delta (\lambda)|^{2\beta} 
\re^{-{\beta \over  g_s} \sum_{i=1}^N V(\lambda_i)},
\ee
where 
\be
\Delta(\lambda_i) =\prod_{i<j}(\lambda_i-\lambda_j)
\ee
is the Vandermonde determinant of the eigenvalues and $V(\lambda)$ is a 
polynomial called the {\it potential} of the 
matrix model. The index $\beta$ takes 
the values $1,1/2, 2$ for Hermitian, real symmetric and symplectic 
matrices, respectively. The {\it Gaussian} matrix model is obtained for a 
potential of the form 
\be
V(\lambda)={1\over 2} \lambda^2.
\ee
In order to obtain generating functions of maps, one finds an asymptotic 
expansion of $Z$ around $N=\infty$. One also sets 
\be
g_s={t\over N}
\ee
where $t$ is the so-called {\it 't Hooft parameter} and is kept finite in 
the expansion. This asymptotic 
expansion is called the large $N$ expansion. Since equivalently we are doing 
the expansion around $g_s=0$, 
we see from (\ref{partitionf}) that the large $N$ expansion is a 
generalization of the standard asymptotic expansions of integrals 
depending on parameters. For the case $\beta=1/2$, which is the one 
considered in \cite{BN,HM}, 
one finds that $F=\log Z$, the so-called {\it free energy}, can be written 
in the form 
\be
F(g_s, t)={1\over 2} F_{\rm o}(g_s, t)+ F_{\rm u}(g_s, t)
\ee
where $F_{\rm o}(g_s, t)$ and $F_{\rm u}(g_s, t)$ have the asymptotic 
expansions
\be
F_{\rm o}(g_s, t)=\sum_{g=0}^{\infty} g_s^{2g-2} F^{\rm o}_g (t) , \qquad 
 F_{\rm u}(g_s, t)=\sum_{r=1/2}^{\infty} g_s^{2r-2} F^{\rm u}_r(t),
\ee
where in the second sum $r$ takes both {\it integer} and {\it half-integer} 
values.  The generating functions $F^{\rm o}_g (t)$ count 
maps on an orientable surface of genus $g$, while $F^{\rm u}_r (t)$ count 
maps on an non-orientable surface with $2r$ crosscaps. 
These generating functions can be computed in closed form by using elegant 
methods started in \cite{BIPZ,BIZ} and which culminated in the 
algebro-geometric formulation of Eynard and collaborators \cite{E,CE,EO}. 
It turns out that all these quantities can be calculated by computing 
residues of meromorphic forms on an algebraic curve of the form
\be
\lbl{scurve}
y(x) = M(x) {\sqrt{\sigma(x)}}, \qquad \sigma(x)= (x-a)(x-b).
\ee
If $V(x)$ is a polynomial of degree $d$, $M(x)$ is also a polynomial, of 
degree $d-2$. This curve is called the {\it spectral curve} of the matrix 
model, and its detailed form depends on the potential $V(z)$ used 
in (\ref{partitionf}). Notice that this curve has a branch cut on $[a,b]$, 
and the branch points $a,b$, as well as $M(x)$ are easy to compute once the 
potential is given (see for example \cite{DGJ,M2} for reviews and examples). 
For a {\em quartic potential} 
\be
V(z)={z^2\over 2} + \lambda z^4
\ee
one has
\be
b=-a= 2\alpha, \quad \alpha^2={1\over 24 \lambda} 
\biggl( -1+ {\sqrt { 1 + 48 \lambda t}}\biggr), 
\ee
and
\be
\label{mquartic}
M(x)= 1+ 8 \lambda \a^2 + 4 \lambda x^2.
\ee
If we set $t=1$ for simplicity, it is easy to see from this description 
that the generating functions $F^{\rm o}_g (\lambda)$, 
$F^{\rm u}_r (\lambda)$ are analytic at 
$\lambda=0$ and that the nearest singularity is at 
\be
\lbl{singularity}
\lambda_c=-{1\over 48}. 
\ee
In terms of the variable $y$ defined by 
\be
y=g_s^{-4/5} \Bigl(1 -{\lambda \over \lambda_c}\Bigr)
\ee
one finds that the generating functions behave as 
\be
g_s^{2g-2} F^{\rm o}_g (\lambda) \sim c^{\rm o}_g y^{-5(g-1)/2}, \qquad 
g_s^{2r-2}F^{\rm u}_g (\lambda) \sim c^{\rm u}_r y^{-5(r-1)/2}, 
\ee
for $g\not=1$. For $g=1$ we rather have
\be
F^{\rm o}_1 (\lambda) \sim c^{\rm o}_1 \log y. 
\ee
The {\it double-scaled} generating functions $F^{\rm o}_{\rm ds}(y)$, 
$F^{\rm u}_{\rm ds}(y)$ are then defined by
\be
\ba
F^{\rm o}_{\rm ds}(y) &= c^{\rm o}_0 y^{5/2} +c^{\rm o}_1 \log y 
+ \sum_{g\ge 2} c^{\rm o}_g y^{-5(g-1)/2}, \\
F^{\rm u}_{\rm ds}(y)&=\sum_{r\ge 1/2} c^{\rm u}_r y^{-5(r-1)/2}.
\ea
\ee
It turns out that these asymptotic expansions can be obtained as solutions 
to ordinary differential equations. For the orientable part, it was shown 
already in 
\cite{BK,DS,GM} that
\be
f(y) =-(F^{\rm o}_{\rm ds})''(y)
\ee
(also called the double-scaled {\it specific heat}) satisfies the 
Painlev\'e I equation
\be
f^2-\frac{1}{3} f''=y.
\ee
In order to compute the coefficients $c_r^{\rm o}$ one considers the formal 
power series for $f$
\be
f(y)=\sum_{n=0}^{\infty} f_n y^{-(5n -1)/2}.
\ee
On the other hand, it was shown in \cite{BN,HM} that the first derivative
\be
g(y)=-2 (F^{\rm u}_{\rm ds})'
\ee
satisfies the differential equation
\begin{equation}
\lbl{eq.fg}
g^3-6gg'+4g''-6gf+6f'=0.
\end{equation}
In order to compute the coefficients $c_r^{\rm u}$ one considers the formal 
power series for $g$ 
$$
g(y)=\sum_{n=0}^\infty g_n y^{-(5n-1)/4}.
$$
The differential equation (\ref{eq.fg}) can be integrated out once and gives 
\cite[Eq.3.25]{BN}
\be
f=-\frac{2}{3}g'+\frac{1}{6}g^2+c \,
\exp\left( \int_1^y \rd y' \, g(y') \right).
\end{equation}
Br\'ezin-Neuberger argue that the appropriate boundary conditions fix $c=0$, 
and it follows that $(f,g)$ satisfy the pair of equations 
\begin{equation}
\lbl{eq.fga}
f^2-\frac{1}{3} f''=y, \qquad
f=-\frac{2}{3}g'+\frac{1}{6}g^2.
\end{equation}
Notice that the second equation is a Riccati type equation for $g$ with $f$ 
known. 

To bring these two equations into the form \eqref{eq.u}, \eqref{eq.v} we must 
make
the change of variables
\begin{equation}
\lbl{eq.yz}
y=2^{\frac{2}{5}}z, \qquad f=2^{\frac{1}{5}} u.
\end{equation}
The change of variables \eqref{eq.yz} forces the following change of variables
\be
\lbl{gf}
g=2^{\frac{3}{5}} v, \qquad \CF_{\rm ds}(z)=2 F_{\rm ds}(y).
\ee
With these change of variables, the differential equations \eqref{eq.fga}
for the pair of functions $(u,v)$ become \eqref{eq.u} and \eqref{eq.v}. 
In the new variables, the specific heat of the matrix model 
\be
-\CF''_{\rm ds}(z) =-{1\over 2} (\CF^{\rm o} _{\rm ds})''(z) 
- (\CF^{\rm u}_{\rm ds})''(z)
\ee
 is given by
\be
\lbl{totsh}
u(z)+v'(z)=\sum_{n=0}^\infty u_n z^{-(5n-1)/2}-
\sum_{n=0}^\infty \frac{5n-1}{4} v_n z^{-(5n+3)/2}
\ee
up to an overall factor $1/2$. Notice that in the second sum, the index $n$ 
is related to the half-genus $g$ of
\cite[p.244]{BC} by
$$
n=2g-1.
$$
Thus, the total specific heat (\ref{totsh}) can be written as
\be
u(z)+v'(z)=\sum_g \biggl(u_g - \frac{5g-3}{2} v_{2g-1} \biggr) z^{-(5g-1)/2}
\ee
where $g$ runs now through the natural numbers and the half natural numbers.
On the other hand, the total heat for this matrix model should 
behave the same way as the generating series 
$$
\sum_g (t_g+p_g) x^g
$$
of \cite{BC}. The relationship (\ref{tgrel}) between $t_g$ and $u_g$ 
forces us to predict that
$$
p_g=\frac{1}{2^{g-2}\Gamma\left(\frac{5g-1}{2}\right)} \frac{5g-3}{2}v_{2g-1}
=\frac{1}{2^{g-2}\Gamma\left(\frac{5g-3}{2}\right)} v_{2g-1}
$$
which is exactly Conjecture \ref{conj.1}. 

We emphasize that this conjecture is a consequence of the matching between counting problems in combinatorics and matrix models. 
Although this relation can be rigorously established in certain cases (see for example \cite{BDG}), the physics techniques to extract the 
coefficients $t_g$, $p_g$ seem to be more efficient than current mathematical techniques, as they lead to simple nonlinear differential equations.

\section{A physics derivation of the Stokes constants}
\lbl{sec.mm}

This section is of independent interest and is included for completeness.
It offers an alternative exact computation (including the Stokes constant
$S'$ of Equation \eqref{eq.stokes})
of the asymptotic expansion of sequences of enumerative interest, 
and although it is not rigorous, it is a motivation of Conjecture 
\ref{conj.2}.  

In many quantum field theories, the standard perturbative expansion in powers 
of 
the coupling constant $g$ has nonperturbative instanton
corrections which behave as $\re^{-k A /g}$, where $A$ is the instanton action 
and 
$k$ is the instanton number. Perturbation theory in powers of $g$
at a fixed instanton number $k$ is in principle possible, and when 
successful, it leads to a computation of the coefficients $a_{k,n}$
of series of the form
\be
\lbl{kinst}
\re^{-k A/g} g^{\gamma} \sum_{n=0}^\infty a_{k,n}g^n. 
\ee
In the physics community it is believed that, at least in some quantum field 
theories, the asymptotic 
expansion of $(a_{k,n})$
for fixed $k$ and large $n$ is exactly computable in terms of the coefficients
$(a_{k',n})$ for $k'$ nearby to $k$ much like Equation \eqref{eq.asall}
of Theorem \ref{thm.munk}. This 
belief  can be mathematically formulated as a resurgence conjecture for the
power series that appear in perturbative quantum field theory. For an example
of this principle for a 3-dimensional quantum field theory, see \cite{Ga}.
One important aspect of the above belief is the ability to compute exactly
all constants (including the Stokes constants) in the asymptotic expansions, 
including the Stokes constants, in some class of quantum field theories.
This is a reflection of integrability of these theories. 

The $1/N$ expansion of gauge theories with gauge group $\mathrm{U}(N)$ 
is in this respect very similar to the standard perturbative expansion in 
powers 
of the coupling constant. The $1/N$ expansion also has nonperturbative 
$k$-instanton
corrections which behave as $\re^{-k A N}$ \cite{KNN,Sh} and have 
the structure
\be
\lbl{kNinst}
\re^{-k N A} N^{\gamma} \sum_{n=0}^\infty \frac{a_{k,n}}{N^n}. 
\ee
In some cases these series should also have resurgence properties and 
the corresponding Stokes constants should be exactly calculable. 

The field theory relevant to the 
current paper is a matrix model, where resurgence properties are certainly 
expected. We can study the behavior of the 
instanton corrections (\ref{kinst}) in this matrix model near the 
singularity (\ref{singularity}) and 
extract as before the most singular part of each coefficient $a_{k,n}$ 
(which in this case will be functions of 
$t, \lambda$). One finds that, in this double-scaling limit, the 
$k$-instanton correction (\ref{kinst}), 
leads to the $k$-th term in the trans-series expansion of the partition 
function. Moreover, one can show 
(at the physics level of rigor) that the one-instanton amplitude gives the 
discontinuity of the first term of the trans-series 
along the Stokes line (see for example \cite{CS} for a detailed exposition), 
and this makes possible to calculate the Stokes constant. 
In the case of matrix models, this was shown in an important paper by F. 
David \cite{Da}, who calculated explicitly 
the Stokes constant $S$ in (\ref{eq.AS}) by using matrix model techniques. 
His calculation was further clarified in \cite{H+} and 
extended in \cite{MSW}. The connection between instanton calculus in matrix 
models and trans-series of differential equations of the Painlev\'e type 
is further discussed in \cite{M1}. In this section, we will compute the 
Stokes constants $S$ and $S'$ by adapting the
matrix model technology of \cite{MSW} to the case of matrix ensembles with 
arbitrary values of $\beta$. Needless to say, the computations in this 
section are not rigorous. 

In computing the asymptotic expansion of the partition function 
(\ref{partitionf}) we have expanded around the saddle point $\lambda_i=0$ 
for 
all the eigenvalues $i=1, \cdots, N$. The first instanton correction 
corresponds to a saddle-point expansion in which one of the eigenvalue 
integrations 
is around a nontrivial saddle point $\lambda =x_0$, and can be written as 
\be
\lbl{onemm}
Z_\beta^{(1)} (N)= {N \over (N!)^{\beta} (2\pi)^{\beta N}} 
\int_{x\in \CI} \rd x \, \re^{-{\beta \over g_s}  V (x)}
 \int_{\lambda \in \CI_0}  \prod_{i=1}^{N-1}\rd \lambda_i\, 
|\Delta (x, \lambda_1, \ldots, \lambda_{N-1})|^{2\beta}\, 
\re^{-{\beta \over g_s} \sum_{i=1}^{N-1} V (\lambda_i)},
\ee
where the first integral in $x$ is over the saddle--point contour around 
the nontrivial saddle-point, which we have denoted by $x\in \CI$, while 
the rest of the $N-1$ eigenvalues are integrated over the saddle--point 
contour $\CI_0$ around the trivial saddle-point. 
The overall factor of $N$ is due to the fact that there are $N$ choices 
for $x$ among the $N$ eigenvalues. We have
\be\lbl{oneinstex}
\ba
Z_\beta^{(1)}(N) &= {N \over (2 \pi N)^{\beta}}  \, Z^{(0)}_{\beta}(N-1) 
\int_{x\in \CI} \rd x \left\langle \det |x {\bf 1} 
- \Lambda_{N-1}|^{2\beta} \right\rangle^{(0)}_{N-1}\, 
\re^{-{\beta\over 2g_s} V(x)} \\
&\equiv {N \over (2 \pi N)^{\beta}} \, Z_{\beta}^{(0)}(N-1) 
\int_{x \in \CI} \rd x\, f(x).
\ea
\ee
\noindent
The notation in this equation is as follows. $Z_\beta^{(0)}(N)$ is the 
partition function (\ref{partitionf}) evaluated around the trivial  
saddle--point $\lambda_i=0$. $\Lambda_{N-1}$ is the diagonal 
$(N-1) \times (N-1)$ 
matrix given by ${\rm diag}(\lambda_1, \cdots, \lambda_{N-1})$. 
$\langle \CO\rangle_N^{(0)}$ is the normalized average of any symmetric 
polynomial $\CO(\lambda_i)$ in the eigenvalues $\lambda_i$, computed by 
a saddle-point calculation around the standard saddle--point, 
\be
\left\langle \CO \right\rangle^{(0)}_N = {\int_{\lambda\in \CI_0} 
\prod_{i=1}^N \rd \lambda_i\, |\Delta(\lambda)|^{2\beta} \, \CO(\lambda)\, 
\re^{-{\beta\over 2g_s} \sum_{i=1}^N V(\lambda_i)} \over 
\int_{\lambda\in \CI_0} \prod_{i=1}^N \rd \lambda_i\, 
|\Delta (\lambda)|^{2\beta} \, \re^{-{\beta \over 2 g_s} 
\sum_{i=1}^N V(\lambda_i)}}.
\ee
Finally, we have also defined 
\begin{equation}\lbl{fx}
f(x) = \left\langle \det |x {\bf 1} - \Lambda_{N-1}|^{2\beta} 
\right\rangle^{(0)}_{N-1}\, \re^{-{\beta\over 2g_s} V(x)}.
\end{equation}
The total partition function will be written as a trans-series expansion
\be
Z_{\beta} =Z_{\beta}^{(0)} + Z_{\beta}^{(1)} +\cdots
\ee
since, as we will see in a moment, $Z_{\beta}^{(1)}$ is exponentially 
small, and it is proportional to the small parameter $\re^{-\CA /g_s}$ 
(where $\CA$ will be calculated shortly). It follows that the trans-series 
expansion of the free energy will be given by 
\be
F=F^{(0)} + F^{(1)} +\cdots
\ee
where 
\begin{equation}\lbl{chempot}
 F^{(1)} = {Z_\beta^{(1)} (N)\over Z_\beta^{(0)}(N)} = 
{N \over (2 \pi N)^{\beta}}\, { Z_{\beta}^{(0)}(N-1) \over 
Z_{\beta}^{(0)}(N)} \int_{x\in \CI} \rd x\, f(x).
\end{equation}
The calculation of this quantity is very similar to the one performed in 
\cite{MSW}, and we will just present the main intermediate steps. The 
result involves 
the connected averages of the matrix model, 
\be\lbl{scor}
W_h (p_1, \ldots, p_h) =\beta^{h-1} \left\langle \tr\, 
{1\over p_1-\Lambda_N} \cdots \tr\, {1\over p_h-\Lambda_N} 
\right\rangle_{(\mathrm{c})}, 
\ee
where the subscript $(\mathrm{c})$ stands for connected. These averages 
have an asymptotic $g_s$ expansion of the form
\be
W_h (p_1, \ldots, p_h) =  \sum_{r=0}^{\infty} g_s^{2r+h-2} W_{r,h} 
(p_1, \ldots, p_h;t),
\ee
where $r$ runs over non-negative integers and half-integers. 
We will also need the integrated version of these averages, 
\be\lbl{theas}
A_{r,h} (x;t) = \left. \int^{x_1} \rd p_1 \cdots \int^{x_h} \rd p_h\, 
W_{r,h} (p_1,\cdots, p_h) \right|_{x_1=\cdots =x_h=x}, \ee
where the integration constant can be fixed by imposing appropriate 
boundary conditions at $x \rightarrow \infty$ (see \cite{MSW} for details). 
Another quantity that enters the computation is the {\it effective potential}
\be
V_{\rm eff}(x;t)=V(x) - 2 A_{0,1}(x;t)
\ee
which satisfies
\be
\lbl{derveff}
{\rd V_{\rm eff}(x;t) \over \rd x}=y(x).
\ee
We can now compute the integral over $f(x)$ at leading order in $g_s$ in 
terms of these quantities. It is easy to see that, due to (\ref{derveff}), 
the non-trivial saddle-points 
are given by the zeros of the moment function $M(x)$. Let $x_0$ be such a 
zero. We obtain 
\be
\int_{x\in \CI} \rd x\, f(x) =  \sqrt{\frac{2 \pi g_s}{\beta 
V''_{\mathrm{eff}} (x_0)}}\, \exp \left( - \frac{\beta}{g_s} 
V_{\mathrm{eff}} (x_0) + \Phi(x_0) \right) \left( 1 +\CO(g_s) \right), 
\ee
where
\be
\Phi(x)= \beta (2 A_{0,2}(x;t) + \partial_t V_{\rm eff}(x;t)) + A_{1/2,1}(x;t)
\ee
The quotient appearing in (\ref{chempot}) can be computed as
\be { Z_{\beta}^{(0)}(N-1) \over Z_{\beta}^{(0)}(N)} ={\Gamma(1+\beta) 
\over \beta^{\beta} (4\pi^2t)^{1-\beta\over 2}}
\exp \Bigl[ -{\beta \over g_s}  \partial_t F_0 +{\beta \over 2}
\partial_t^2 F_0 - \partial_t F_{1/2}  +\CO(g_s) \Bigr]. 
\ee
Putting everything together, we obtain the contribution of $x_0$ to 
$F^{(1)}$ at leading order in the $g_s$ expansion, 
\be
\lbl{fxo}
F^{(1)}_{x_0} ={N \over (2 \pi N)^{\beta}} {\Gamma(1+\beta) \over 
\beta^{\beta} (4\pi^2t)^{1-\beta\over 2}} 
\sqrt{\frac{2 \pi g_s}{\beta V''_{\mathrm{eff}} (x_0)}}\, 
\exp \left( - \frac{\beta}{g_s} \CA \right) 
\exp\Bigl[ \Phi(x_0)+{\beta \over 2}\partial_t^2 F_0 - 
\partial_t F_{1/2} \Bigr]
\ee
where \cite{MSW}
\be
\CA=V_{\mathrm{eff}} (x_0) + \partial_t F_0= \int_b^{x_0} y(x) \rd x
\ee
is the instanton action of the matrix model corresponding to the nontrivial 
saddle-point at $x_0$. Given a matrix model potential, the leading 
contribution to $F^{(1)}$ is given by the sum of the contributions of 
the saddle-points with 
lowest instanton action (in absolute value).

All the quantities appearing in (\ref{fxo}) which are needed when 
$\beta=1$ have been 
already computed explicitly in \cite{MSW}. Using the results of \cite{CE} 
we can also compute all the quantities needed for $\beta\not=1$, in terms 
of data of the spectral curve 
(\ref{scurve}). Since $M(z)$ is a polynomial, it can be written as 
\be
M(z)=c\prod_{i=1}^{d-2} (z-z_i).
\ee
We find for example
\be
\label{woh}
W_{1/2,1}(p)=(1-\beta^{-1}) \biggl\{ {d-1\over 2{\sqrt{\sigma(p)}}} 
-{1\over 4} {2p-a-b\over (p-a)(p-b)}-{1\over 2 {\sqrt{\sigma(p)}}}
\sum_{i=1}^{d-2} \biggl( {{\sqrt{\sigma(p)}} -{\sqrt{\sigma(z_i)}} 
\over p-z_i} \biggr)\biggr\}, 
\ee
and
\be
\label{foh}
\partial_t F_{1/2} =-(1-\beta) \biggl\{ {1\over 2} \log \biggl[ {1\over t} 
\Bigl( {b-a\over 4 }\Bigr)^2 \biggr] + \log c + \sum_{i=1}^{d-2} \log 
\biggl[ {1\over 2} \Bigl(z_i -{a+b\over 2}  + {\sqrt{\sigma(z_i)}}\Bigr) 
\biggr] \biggr\}.
\ee
In the expression for $\partial_t F_{1/2}$ we have subtracted $\partial_t 
F^G_{1/2}$, which is the $r=1/2$ free energy of the Gaussian model. 

Let us specialize these results for the quartic matrix model considered in 
\cite{BN,HM}. From (\ref{mquartic}) we see that $M(x)$ has two zeros at 
$\pm x_0$, 
\be
\lbl{qzero}
x_0^2=-{1+ 8 \lambda \a^2 \over 4 \lambda}.
\ee
Setting $\beta=1/2$ one finds by integration of (\ref{woh}) that 
\be
\ba
\label{aoh}
A_{1/2,1}(p)&={1\over 4} \log (p^2-4 \a^2) -{3\over 2} \log \Bigl[  p 
+  {\sqrt{p^2-4 \a^2}}\Bigr] \\
&+{1\over 2} \log\biggl[ 4\a^2 (x_0^2 +p^2) -2 x_0^2 p^2  -2 x_0^2 p^2 
{\sqrt{(1-4\a^2/p^2)(1- 4 \a^2 /x_0^2)}}\biggr] \\
&+{3\over 2} \log 2 -{1\over 2} \log \Bigl[ 4 \a^2 -2 x_0^2 -2 x_0^2 
{\sqrt{(1- 4 \a^2 /x_0^2)}} \Bigr].
\ea
\ee
and from (\ref{foh}) that
\be
\partial_t F_{1/2}=-{1\over 4} \log {\a^2 \over t} -\log \biggl[ {1\over 2} 
\Bigl(
{\sqrt{1+ 8 \lambda \a^2}} +{\sqrt{1+ 24 \lambda \a^2}}\Bigr)\biggr].
\ee

Before proceeding with the calculation of (\ref{fxo}), we present a test of 
these expressions. Based on 
general arguments relating matrix integrals and enumeration of maps 
\cite{E2,BDG} 
it follows that 
\be
\langle \tr \Lambda_N^4\rangle_{\BR \BP^2} ={\rm Res}_{p=0} \, 
p^4 W_{1/2,1}(p) 
\ee
is a generating functional for the number of rooted quadrangulations 
$c_n$ of the projective plane with 
$n$ vertices,
\be
\langle \tr \Lambda_N^4\rangle_{\BR \BP^2} =t^2 \sum_{n=1}^{\infty} 
c_n (-4 \lambda t)^{n-1}. 
\ee
This generating functional can be explicitly obtained from (\ref{aoh}) in 
terms of $\alpha$ and $x_0$,
\be
\langle \tr \Lambda_N^4\rangle_{\BR \BP^2}=x_0^4 - 
\alpha^4 -x_0^2 (x_0^2 + 2 \alpha^2) {\sqrt{1-{4 \alpha^2 \over x_0^2}}}. 
\ee
One finds for the $c_n$ sequence the values
\be
5, \, \, 38, \, \, 331, \, \, 3098, \, \, 30330, \, \, 306276, \, \,  
3163737, \, \, \cdots
\ee
The first ones can be tested with the results of \cite{A+}. It should be 
also possible to obtain the generating functional for these numbers from 
the general combinatorial results of \cite{G3}. 

The zeros (\ref{qzero}) of $M(x)$ give two saddle-points with the same 
value of $\CA$ and the same $F^{(1)}_{\pm x_0}$, therefore they both 
contribute to the instanton amplitude. In order to make contact with the 
trans-series solution of the differential equation, we must analyze the 
behavior of (\ref{fxo}) near the singular point (\ref{singularity}) 
(we set again $t=1$). 
Using the variable $z$ introduced in (\ref{eq.yz}) and taking into account 
the factor of $2$ relating $F$ and $\CF$ in (\ref{gf}), one obtains
\be
\ba
\quad \CF^{(1)}_{\beta=1}(z)&={\ri \over 8\cdot 3^{3\over 4} 
\pi^{1\over 2}} z^{-5/8} \exp\Bigl( -{8 {\sqrt{3}} \over 5} z^{5/4}\Bigr) 
+\cdots\\
\quad \CF^{(1)}_{\beta={1\over 2}}(z)&={\ri \over {\sqrt{2}}} 
\exp\Bigl( -{4 {\sqrt{3}} \over 5} z^{5/4}\Bigr) +\cdots.
\ea
\ee
From here, and taking into account the relations between $\CF(z)$ and 
$u(z)$, $v(z)$, we find
\be
u_1(z)=-(\CF^{(1)}_{\beta=1})''(z)= -\ri {3^{1\over 4} 
\over 2{\sqrt{\pi}}} z^{-5/8} \exp\Bigl( -{8 {\sqrt{3}} \over 5} 
z^{5/4}\Bigr)\bigl(1 +\CO(z^{-5/4})\bigr)
\ee
and 
\be
v_1(z)=-2 (\CF^{(1)}_{\beta={1\over 2}})'(z)=\ri {\sqrt{6}} \, z^{1\over 4} 
\exp\Bigl(-{4 {\sqrt{3}} \over 5} z^{5/4}\Bigr)\bigl(1 +\CO(z^{-5/4})\bigr)
\ee
We now recall that this computation computes the discontinuity across the 
Stokes line, therefore the overall constants appearing in these 
two trans-series solutions are the Stokes constants $S$, $S'$ that we 
introduced before. Let us end this section with a problem.

\begin{problem}
\lbl{prob.5}
Compute rigorously the $k$-instanton expansion of a matrix model (with
polynomial potential) using the Riemann-Hilbert method or the method
of Borel transforms.
\end{problem}

\section{Numerics}
\lbl{sec.numerics}
We now give numerical evidence for Conjecture \ref{conj.2}. 
To do this, we study the sequence
\be
s_n= {2 \pi (A/2)^n \over \Gamma(n)} v_n
\ee
which according to \ref{thm.1} has the following asymptotic behavior
\be
\label{sequ}
s_n \sim \sum_{k=0}^{\infty} {a_k \over n^k}, \qquad a_0 =-\ri S', 
\quad a_1 =-\frac{\ri}{2} S' A \nu_1, \cdots 
\ee
In order to test Theorem \ref{thm.1} and Conjecture \ref{conj.2}, it is 
useful to have a precise numerical method to determine the numbers $a_k$. 
The method of {\em Richardson transforms} 
(see for example \cite[p.375]{BS} for an 
exposition) is very well-suited for sequences of the type (\ref{sequ}) 
and it is easy to implement. Given a sequence of the form (\ref{sequ}), its 
$N$-th Richardson transform is defined by 
\be
s^{(N)}_n=\sum_{k=0}^N{s_{n+k} (g+k)^N(-1)^{k+N}\over k!(N-k)!}.
\ee
The effect of this transformation is to remove subleading tails in 
(\ref{sequ}). The values $s^{(N)}_n$ give numerical approximations to $a_0$, 
and these 
approximations become better as $N$, $n$ increase. As an example of this 
procedure, we plot the values of the sequences $S^{(N)}_n$, for $N=0,1,5$ and 
$n=1, \cdots, 250$. The bottom curve corresponds to $N=0$, the top curve 
corresponds to $N=1$, and the intermediate curve is $N=5$. 
\vskip .5cm
$$
\psdraw{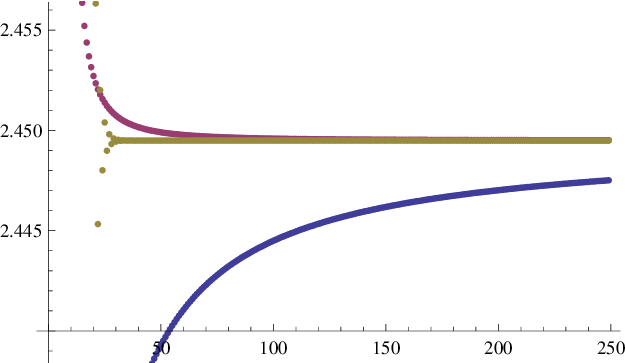}{3in}
$$
\vskip .5cm
We can see that the Richardson transforms provide a very fast convergence 
to the expected value 
\be
{\sqrt{6}}=2.44948974278317809819728407471....
\ee
Numerically, we find that $n=250$ and $N=20,30$ provide an approximation to 
the expected value with $28$ (respectively, $30$) significant digits,
\be
\ba
s^{(20)}_{250}&= 2.44948974278317809819728407459..., \\
s^{(30)}_{250}&= 2.44948974278317809819728407471...
\ea
\ee
We can verify in a similar way the value of $a_1$ by considering the 
auxiliary sequence
\be
r_n = n\Bigl( \ri {s_n\over  S'}-1\Bigr)
\ee
with the asymptotic behavior
\be
r_n \sim \frac{1}{2} \nu_1 A+\CO(1/n), \qquad n\rightarrow \infty. 
\ee
We plot the values of the sequences $r^{(N)}_n$, for $N=0,1,5$ and 
$n=1, \cdots, 250$.
\vskip .5cm
$$
\psdraw{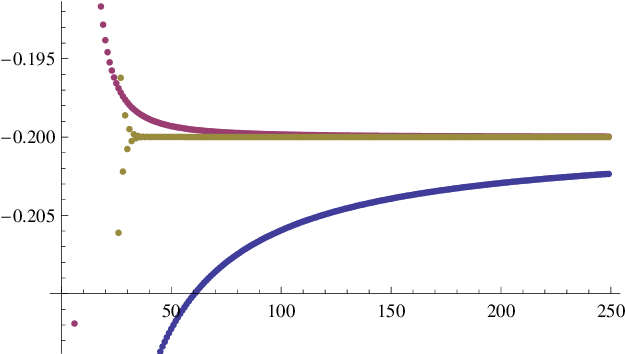}{3in}
$$
Like before, we have very fast convergence to the expected value
\be
 \frac{1}{2}\nu_1 A=-{1\over 5}. 
 \ee
One finds numerically
\be
\ba
r^{(20)}_{250}&=-0.200000000000000000000000001520...,\\
r^{(30)}_{250}&=-0.200000000000000000000000000002....
\ea
\ee
\begin{remark} 
\lbl{rem.secondS}
We can also study the asymptotics of the coefficients $v_{n,k}$ appearing in 
the trans-series solution (\ref{eq.transv}). This asymptotics is governed by 
Equation \eqref{eq.asall}, where $\lambda=A/2$, $\beta=0$, and $S_1=S'$. Based on our 
numerical results, we conjecture that
\be
S_{-1}=-\ri {{\sqrt{6}}\over 12}. 
\ee
\end{remark}

%%%%%%%%%%%% Mathecita file: Unorgravity.nb
Finally, we end with the following expectation of Conjecture \ref{conj.1}:
$$
p_{\frac{41}{2}}=\frac{
1238878081129358302459331398309144842472024202171957968278854904568087
551305256373}{10986030082548950321157435333449889551411576832 \sqrt{6}
    \Gamma\left(\frac{199}{4}\right)}.
$$

\ifx\undefined\bysame
        \newcommand{\bysame}{\leavevmode\hbox
to3em{\hrulefill}\,}
\fi

\end{document}